\documentclass[12pt,article]{amsart}
 \usepackage{amssymb,amsthm}

\usepackage{amsthm}
\usepackage{amsmath}
\usepackage{amssymb}
\usepackage{comment}
\usepackage{graphicx}
\usepackage{color}
\usepackage{enumerate,array}
\usepackage{leqno}
\usepackage{hyperref}
\hypersetup{colorlinks,citecolor=blue,linkcolor=blue,urlcolor=blue}
\usepackage[utf8]{inputenc}

 \newtheorem{thm}{Theorem}[section]
  \newtheorem{theorem}{Theorem}[section]
 \newtheorem{cor}[thm]{Corollary}
 
  \newtheorem{lemma}[thm]{Lemma}
 \newtheorem{prop}[thm]{Proposition}
 \theoremstyle{definition}
 
 \theoremstyle{remark}

\theoremstyle{remark} 
\newtheorem{remark}{Remark} 

\def\cHm1{\mathcal{H}^{-1}(\Omega)}
\def\cL2{\mathcal{L}^2(\Omega)}

\newcommand{\R}{\mathbb{R}}

\def\cHd12{\mathcal{H}^{\frac{1}{2}}(\partial\Omega)}
\def\cLd2{\mathcal{L}^{2}(\partial\Omega)}

\begin{document}
\title[Time-scale analysis non-local diffusion systems]{Time-scale analysis non-local diffusion systems, applied to disease models}
\author{M. C. Pereira$^1$, S. Oliva$^2$, L. M. Sartori$^3$}
\address{$^1,^2,^3$ Dept. Mat. Aplicada, IME, USP, Rua do Mat\~{a}o, 1.010, CEP 05508-900, S\~{a}o Paulo, SP, Brazil}

\email{marcone@ime.usp.br} \email{smo@ime.usp.br} \email{larissa@ime.usp.br}

\keywords{reaction-diffusion equations, nonlocal systems, Neumann problem, epidemiology models.\\
\indent 2010 {\it Mathematics Subject Classification.} Primary 35K57; Secondary 92B05.}

\begin{abstract}
The objective of the present paper is to use the well known Ross-Macdonald models as a prototype, incorporating spatial movements, identifying different times scales and proving a singular perturbation result using a system of local and non-local diffusion. This results can be applied to the prototype model, where the vector has a fast dynamics, local in space, and the host has a slow dynamics, non-local in space.
\end{abstract}

\maketitle

\section{Introduction}\label{sec:Intro}

In this work we are interested in the dynamics of vector-borne diseases. These  dynamics, due to the interaction between hosts and vectors, behave quite different from direct diseases. In fact, they behave, as we will show, more likely to direct diseases with nonlinear incidence rates (see \cite{Ruan}). It has been a challenge for scientists and public health officers to predict outbreaks of such diseases and, in some countries, diseases like dengue and malaria are a leading cause of serious illness and death among children. Another point is that such illness, due to globalization, are spreading all over the world. For instance,  dengue is currently the human viral disease with the highest number of cases, being an arbovirus of the family \textit{Flaviviridae},
genus \textit{Flavivirus}, is transmitted through the bite of female mosquitoes of the genus \textit{Aedes} infected with the virus,
which are also responsible for the transmission of Zika, Chikungunya and Yellow Fever virus 
\cite{wang2017}.

Dengue is estimated to be endemic in more than 100 countries, where climate favors the proliferation of vectors,
and  approximately half of the world's population is at risk of contracting the disease
\cite{liu2016climate, kraemer2015global, shepard2016global, rodriguez2011re}.
Due to lack of vaccination, basic sanitation, climate changes, and with increasing human mobility, such diseases
 are spreading and appearing in new regions. The host population can become infected in environments that are not their places of residence, and since mosquitoes do not travel great distances, humans can carry the disease to different locations where there are susceptible mosquitoes, and this may also lead to increased population heterogeneity and consequently in changes of the disease dynamics \cite{massad2008scale, amaku2016magnitude, dosSantos2018}.
 Thus, it is of interested that we have reliable models that can predict the spread of outbreaks trying to incorporate space heterogeneity, human and vector dynamics.

For dengue, the current control measures are the control of the vector population and its breeding sites, with the
use of insecticides, adulticides and population awareness campaigns. More recently, it also been  successufully tested control measures with \textit{Wolbachia} bacteria, which prevents the vector from transmitting
 \cite{boccia2014, JKing2018}. Some vaccines have been
tested and others are in the testing phase, but the great difficulty is that such vaccines should be tetravalent, that is,
be effective against the 4 existing serotypes of the disease \cite{precioso2015clinical, maier2017analysis}.

Mathematical models were improved to analyze the dynamics of dengue propagation and to evaluate the best strategies
of control. Studies with spatial networks, or meta-populations,  provide a way to understand the interactions between individuals in different
scales, being a powerful tool to understand the characteristics of transmission in communities, regions and countries
incorporating spatial heterogeneity \cite{massad2008scale, amaku2014, iggidr2017vector}.
In addition, as in  modelling the dynamics of several vector borne diseases, if the goal is to fit the model to real data,  one has to deal with  the asymptomatic cases,
reliable data, in particular for the mosquitoes population, besides having to take into account the different  times scale of  vectors and hosts, which makes it difficult to study and understand the dynamics of the disease \cite{rocha2013time, massadestimating}.

Our approach will be deterministic, we will not take into account stochastic effects or incorporate the element of chance in the models. The prototype model here is continuous in the space domain, but a lot of work has been done considering discrete networks in space, which will provide system of ordinary differential equations. There are advantages and disadvantages to both approaches. From the mathematical point of view, there are several theoretical challenges in the continuous model.

One of the main questions that public health officers, and thus modelling in epidemiology,  concern the  global stability of equilibria,
since this characterizes if a disease will become endemic. This will, from the epidemiology side, be characterized by the basic reproduction number, $R_0$, being larger or smaller than one. On the other hand, from the mathematical point of view, this is characterized by the existence of a stable equilibrium point, with positive number (or density) of infected individuals. This characterization is crucial for predicting outbreaks. In practice, once we propose a model to predict the outbreak, one has to fit the parameters for a specific disease, this is a real challenge when  vectors are involved since is very difficult to estimate its population, a crucial parameter in most problems \cite{massadestimating}. 

We will propose a host-vector disease compartmental model that will include space and time dynamics, thus capturing the heterogeneity of hosts and vectors in space. This is quite a challenge since hosts dynamics, in medium scales, are very difficult to model. The purpose is to follow the ideas of multiscale dynamics, in order to simplify the system modelling the vectors and hosts dynamics. We will consider that vector population dynamics is much faster than the hosts dynamics. This model can be applied, for instance, for dengue (see \cite{marconelarissaoliva}). 

There are several results dealing with singular perturbation using several different diffusion operators, one can refer to \cite{magal,Henry,hale}. But here we will couple a non-local operator to a local one. We will consider a vector-borne disease modeled by the  Ross-Macdonald model incorporating spatial movements, both for host and vector population. We consider that hosts can move non-locally and vectors can move locally, this will lead us to the following prototype that mix different infinite dimensional operators in the same system.

\section{Setting the model}

Being  more precise, the model we employed to describe the dynamics of the disease transmission considers that the 
total host population ($N_{h}$) is divided in susceptible ($S$) and infected ($I$)
 and it is coupled with the compartments of susceptible $S_{m}$ and infected $I_{m}$ vectors 
with total population given by $N_{m}$, the model is named $SIS_{m}I_{m}$. Thus, we describe the 
interaction dynamics between the compartments through a system of ordinary differential equations (ODEs):
\begin{eqnarray*}
\begin{array}{lllll}
dS/dt = \mu_{h}(N_{h}-S) - \beta SI_{m}/N_{m} \\
dI/dt = \beta SI_{m}/N_{m} - (\gamma + \mu_{h})I \\
dS_{m}/dt = \mu_{m}(N_{m}-S_{m}) - \omega S_{m}I/N_{h} \\
dI_{m}/dt = \omega S_{m} I/N_{h}  - \mu_{m}I_{m}
\end{array}
\label{eq:SIRSmIm}
\end{eqnarray*}
where $\mu_{h}$ is the birth/mortality rate of hosts, $\beta$ and $\omega$ are the transmission 
rates from vectors to hosts and from hosts to vectors, respectively, $\gamma$ is the recovery 
rate of hosts and $\mu_{m}$ is the vector birth/mortality rate. 

Assuming that birth and mortality rates are equal, we have that populations remain constant 
over time, that is, $N_{h}(t) = S(t) + I(t)$ and $N_{m}(t) = S_{m}(t) + I_{m}(t)$, 
consequently we can easily obtain $S(t) = N_{h}(t) - I(t)$ and $S_{m}(t) = N_{m}(t) - I_{m}(t)$ 
and then work with an equivalent reduced system:
\begin{eqnarray*}
dI/dt &=& \beta (N_{h}-I)I_{m}/N_{m} - (\gamma + \mu_{h})I \\
dI_{m}/dt &=& \omega (N_{m}-I_{m}) I/N_{h}  - \mu_{m}I_{m}
\end{eqnarray*}

We also consider that the hosts and vectors dynamics are in different scales given by the order of the birth/mortality and transmission rates.
Hence, to describe this time scales separation, we add the singular term $1/\varepsilon$ which leads us to the following system
\begin{eqnarray*} \label{s1epsilon}
dI/dt &=& \beta (N_h-I) I_{m}/N_{m} - (\gamma + \mu_{h})I \\
dI_{m}/dt &=& \frac{1}{\varepsilon} \left(\overline{\omega} (N_{m}-I_{m}) I/N_{h} - \overline{\mu_{m}}I_{m} \right).
\end{eqnarray*} 

Moreover, letting $i=I/N_h$, $j=I_m/N_m$, $\alpha_h=\beta$, $\beta_h=\gamma+\mu_h$, $\alpha_v=\overline{\omega}$, $\beta_v=\overline{\mu_m}$ and  rewriting the parameters, we get
\begin{eqnarray}
\label{odesis}
\begin{array}{lll}
\displaystyle di/dt &=& \displaystyle\alpha_h(1-i)j-\beta_hi  \\
\displaystyle dj/dt &=& \displaystyle \frac{\alpha_v}{\varepsilon} (1-j) i - \frac{\beta_v}{\varepsilon} j.
\end{array}
\end{eqnarray} 
In this way, see for instance \cite{rocha2013time}, we set a system where the vector population dynamics is much faster than the hosts one as $\varepsilon \approx 0$.

As for the spatial mobility, we will consider a regular bounded space domain $\Omega \subset \R^n$ with exterior unit normal $\vec{n}$. The spatial movement for the vector will be modeled by the usual Laplacian operator with Neumann boundary condition ($\Delta$) and the hosts spatial dynamics will be modeled by an non-local operator  $K_J$ defined as follows
\[
\displaystyle K_J\:i(x)=\int_\Omega J(x-y)(i(y)-i(x))dy, \quad x \in \Omega.
\]
Along whole paper we assume that the kernel $J$ satisfies the hypotheses 
$$
{\bf (H_J)} \qquad 
\begin{array}{c}
J \in \mathcal{C}(\R^n,\R) \textrm{ is non-negative with } J(0)>0, \\ J(-x) = J(x) \textrm{ for every $x \in \R^n$ and } \\
\int_{\R^n} J(x) \, dx = 1.
\end{array}
$$

Under these conditions, the $K_J$ is known as a nonlocal operator with non-singular kernel and Neumann condition \cite{libro}. 
Putting the local disease dynamics (\ref{odesis}) with the spatial dynamics, we get our main model with Neumann boundary condition and $d_1,d_2>0$,
\begin{equation} \label{SIRUV}
\displaystyle\left\{
\begin{array}{l}
\displaystyle \frac{\partial i}{\partial t}=\alpha_h(1-i)j-\beta_hi+d_1K_J\: i,\\
\\
\displaystyle\frac{\partial j}{\partial t}=\frac{\alpha_v}{\varepsilon}(1-j)i-\frac{\beta_v}{\varepsilon}j+d_2\Delta j,\\
\end{array}
\right. \quad x \in \Omega, \; t>0 
\end{equation}
\begin{equation}\label{Neuman}
\frac{\partial j}{\partial {\vec{n}}}=0, \quad x \in \partial \Omega.
\end{equation}

The paper is organized as follows. In Section \ref{sec-asym}, we use asymptotic expansion approach to find a limit equation to \eqref{SIRUV}. Indeed, we obtain a limit model which represents the original system in an effective way as $\varepsilon$ goes to zero. 
We also discuss some properties to the limit equation, such as, conditions to guarantee the existence of a positive stationary solution globally stable. 
Next, in Section \ref{sec-conv}, we consider a more general system, which includes our prototype model \eqref{SIRUV}, showing convergence at $\varepsilon = 0$. Assuming appropriated assumptions, we prove convergence in $L^2(\Omega)$ spaces in finite intervals of time. Finally, we make some comments about the dynamics of the prototype model as a consequence of our estimates in Section \ref{appl}.

\vspace{3mm}

\section{Asymptotic Expansion} 
\label{sec-asym}

In this section we use power series expansion to analyze in a formal way the asymptotic behavior of the  singular perturbed system (\ref{SIRUV}) with respect to parameter $\varepsilon>0$. We assume functions $i$ and $j$ satisfy  
\[\displaystyle i=i_0+\varepsilon i_1+\dots
\quad \textrm{ and } \quad 
\displaystyle j=j_0+\varepsilon j_1+\ldots\]
Thus, the time derivatives themselves set 
$$
\begin{gathered}
\frac{di}{dt} =\frac{d i_0}{dt} + \varepsilon \frac{di_1}{dt} + \dots \quad \textrm{ and } \quad 
\frac{dj}{dt} = \frac{dj_0}{dt} + \varepsilon \frac{dj_1}{dt} + \ldots
\end{gathered}  
$$
which gives us from the right-hand side of \eqref{SIRUV} that 
\begin{eqnarray*}
\frac{di}{dt} & = & \left[    \alpha_h(1 - i_0)j_0 - \beta_h i_0 + d_1 K_J i_0  \right] \\
& & \qquad + \varepsilon \left[   \alpha_h(j_1 - i_0 j_1 - i_1 j_0 )  - \beta_h i_1 + d_1 K_J i_1  \right] + O(\varepsilon^2) \\
\frac{dj}{dt} & = & \left[    \alpha_v(1 - j_0)i_0 - \beta_v j_0  \right] \\
& & \qquad  + \varepsilon \left[   \alpha_v(i_1 - j_0 i_1 - j_1 i_0 )  - \beta_v i_1 + d_2 \Delta j_0  \right] + O(\varepsilon^2)
\end{eqnarray*}

Hence, if we plug these expressions in the system \eqref{SIRUV}, we get at $\varepsilon=0$ 

$$
\displaystyle\left\{
\begin{array}{rl}
\displaystyle \frac{\partial i_0}{\partial t}&=\alpha_h(1-i_0)j_0-\beta_hi_0+d_1K_J \:i_0,\\
\\
\displaystyle 0&=\alpha_v(1-j_0)i_0 - \beta_v j_0.\\
\end{array}
\right.
$$
Consequently, we obtain  
\[\displaystyle
j_0=m(i_0)=\frac{\alpha_vi_0}{\alpha_vi_0+\beta_v}
\]
and then, we deduce the reduced equation 
\begin{equation} \label{limin}
\displaystyle \frac{\partial i_0}{\partial t}=\alpha_h(1-i_0)m(i_0)-\beta_hi_0+d_1K_J\: i_0
\end{equation}
with initial condition $i_0(t,x) = i_0(0,x)$.

It will be seen in Section \ref{sec-conv} that the solutions $i$ of \eqref{SIRUV} can be approximated to the functions $i_0$ given by \eqref{limin}. 
Indeed, at $\varepsilon = 0$ we have
\[\displaystyle
i    \approx  i_0
\]
under appropriated functional spaces, initial conditions and finite time.

\subsection{The limit problem}

Let us now discuss a little the limit equation \eqref{limin}.
First, we notice that \cite[Theorem 3.2]{HMMV} implies that the Strong Maximum Principle works to \eqref{limin} in the space of non-negative continuous functions in $\bar \Omega$ which we denote here by $\mathcal{C}(\Omega)$.
Hence, since $K_J$ is zero in any constant function, we have that the nonlocal equation \eqref{limin} defines a dynamical system which behaves as the solutions of the ordinary differential equation 
\begin{equation} \label{odec}
\frac{d z}{d t}=\alpha_h(1-z)m(z)-\beta_h z.
\end{equation}

Indeed, problem \eqref{limin} possesses two constant equilibria, the null constant and 
$$
i^*_0=\frac{\alpha_h \alpha_v -\beta_h \beta_v}{\alpha_h+\alpha_v \beta_h}.
$$
Thus, under the additional condition 
$$
{\bf (H_C)} \qquad \qquad 
\alpha_h \alpha_v >\beta_h \beta_v 
$$
we conclude that the null function is an unstable equilibrium to \eqref{limin} and $i^*_0$ is globally stable. In fact, we have the following result:

\begin{prop}
Let us assume condition ${\bf (H_J)}$.
\begin{enumerate}
\item[a)] Suppose $i_0(0,x)$ continuous and non-negative. Then \eqref{limin} possesses a continuous, non-negative solution for all $t>0$.
\item[b)] Assume still condition ${\bf (H_C)}$. Then, the positive constant $i^*_0$ is the unique stationary and positive solution to \eqref{limin} which is globally stable in $L^\infty(\Omega)$ for any solution with non-trivial and non-negative initial condition in $\mathcal{C}(\Omega)$.
\end{enumerate}
\end{prop}
\begin{proof}
The existence and uniqueness of the globally stable equilibrium is a consequence of Maximum Principle arguments shown at \cite[Theorem 3.2]{HMMV} being minor amendments of the proofs of the classical theory discussed in \cite{Fife, Pao}. See also \cite[Theorem 3.6]{HMMV} and \cite[Exercise 8]{Henry}. On the other hand, the convergence in $L^\infty(\Omega)$ of the solutions follows from comparison with the solutions of the ODE \eqref{odec}.
\end{proof}

\begin{remark} \label{0stable}
It is not difficult to see that, if $\alpha_h \alpha_v \leq \beta_h \beta_v$, then the null function is the unique stationary and non-negative solution of \eqref{limin} which is globally stable for any solution with non-negative initial condition.
\end{remark}

\section{Convergence Results}  \label{sec-conv}

In this section, we estimate the convergence of the solutions in a more general framework. We analyze the following singularly perturbed system 
\begin{equation} \label{equa}
\left\{
\begin{gathered}
\dot x = f(x,y) + K_Jx \\
\varepsilon \dot y = g(x,y) + \varepsilon \Delta y \\
\end{gathered}
\right. \quad \textrm{ in } \Omega, \quad \varepsilon >0, 
\end{equation}
with homogeneous Neumann boundary condition
\begin{equation}  \label{bcond}
\frac{\partial y}{\partial N} = 0 \quad \textrm{ on } \partial \Omega.
\end{equation}

As before, we suppose $\Omega \subset \R^n$ is a regular bounded domain, $\Delta$ is the Laplacian differential operator and $K_J$ is the nonlocal one
$$
K_J x(u) = \int_\Omega J(u-v) ( x(v) - x(u) ) dv, \quad u \in \Omega.
$$

The nonlinearities $f$ and $g:\R^2 \mapsto \R$ are smooth functions and will include the class of those ones discussed in the previous sections.

We show that in the limit $\varepsilon \to 0$ the slow component $x(t)$ converges to a function $X(t)$ which is governed by the effective equation 
\begin{equation} \label{limit}
\dot X = f(X,m(X)) + K_JX, \quad \textrm{ in } \Omega,
\end{equation}
where $y=m(x)$ is the graph representation of a set given by 
\begin{equation} \label{funcm}
g(x,m(x))=0.
\end{equation}

Under $f$ and $g$ we set the following conditions:

${\bf(H_{fg})}$ There exist positive constants $M$, $N$, $\gamma$, $\alpha$, $\beta$ and $\delta$ such that 
for all $x$ and $y$ satisfying $0 \leq x \leq N$ and $0 \leq y \leq M$, we have 
\begin{enumerate}
\item[(i)] $|f(x,y)|$ and $|g(x,y)|$ uniformly bounded for a constant $k>0$;
\item[(ii)] 
$$
\frac{\partial f}{\partial x}(x,y) \leq - \alpha 
\quad  \textrm{ and }  \quad  0 \leq \frac{\partial f}{\partial y}(x,y) \leq \beta N;
$$
\item[(iii)] 
$$
\frac{\partial g}{\partial y}(x,y) \leq - \delta
\quad \textrm{ and } \quad 0 \leq \frac{\partial g}{\partial x}(x,y) \leq \gamma M.
$$
\end{enumerate}

${\bf(H_{\infty})}$ Next, we assume the nonlinearities $f$ and $g$ are such that :
\begin{enumerate}
\item[(i)] there exists a constant $\hat \rho \in \R$ in such way that $(f(x,y), g(x,y)) + \hat \rho \, (x,y)$ is an increasing function;
\item[(ii)] $( f(x,y), \epsilon^{-1} g(x,y) ) \cdot (x,y)<0$ wherever $(x,y)$ does not belong to the rectangle 
$$\mathcal{R} = \{ (u,v) \in \R^2 \, : \, 0\leq u \leq N \textrm{ and } 0 \leq y \leq M \}.$$
\end{enumerate}

Finally, we suppose function $m$ given by \eqref{funcm} satisfies 
$$
{\bf (H_m)} \qquad 
\begin{gathered}
0 \leq m(x) \leq \frac{\gamma}{\delta}M N 
\quad \textrm{ and } \quad 
0 \leq m'(x) \leq \frac{\gamma M}{\delta} 
\end{gathered}
$$
whenever $0 \leq x \leq N$.

\begin{remark} \label{Linfty}
Due to condition ${\bf (H_\infty)}$, it follows from Maximum Principle arguments, discussed for instance in \cite[Lemma 3.11]{libro} and \cite[Exercise 8]{Henry}, the global existence and uniformly boundedness of solutions to \eqref{equa} in $L^\infty(\Omega)$ for all $\epsilon>0$ wherever the initial condition $(\|x_0\|_{L^\infty}, \|y_0\|_{L^\infty}) \in \mathcal{R}$. Indeed, we can use comparison to show that problem \eqref{equa} defines a dynamical system as the ODE
 \begin{equation*} \label{eq}
\left\{
\begin{gathered}
\dot x = f(x,y) \\
\dot y = \epsilon^{-1} g(x,y) 
\end{gathered}
\right.
\end{equation*}
since their solutions also satisfy the boundary condition \eqref{bcond}.
\end{remark}

\begin{remark}
We notice that nonlinearities $f(x,y)=\beta (N-x)y- \alpha x$ and $g(x,y)=\gamma (M-y)x - \delta y$ with $M$, $N$, $\alpha$, $\beta$, $\gamma$ and $\delta$ positive constants satisfy conditions ${\bf (H_{fg})}$, ${\bf (H_\infty)}$ and ${\bf (H_m)}$. 
In this way, \eqref{equa} can be seen as a generalization of problem \eqref{SIRUV}.
\end{remark}

\subsection*{Estimates and Convergence}

Let $(x,y)$ be the solution of \eqref{eq}, and let us introduce the deviation of the fast variable from the invariant set  
$$
y=m(x)+\eta.
$$ 
First, we estimate the rate at which $\eta$ goes to zero.
Since 
$$
\dot \eta = \dot y -m'(x) \dot x
$$
we can rewrite system \eqref{eq} as 
\begin{equation}    \label{siseta}
\left\{
\begin{array}{lll}
\displaystyle \dot x &=& f(x,m(x) + \eta) + K_Jx \\
\displaystyle \dot \eta &=& {\varepsilon}^{-1} g(x, m(x)+\eta) + \Delta( m(x) + \eta) \\
\displaystyle && \qquad \qquad - m'(x) \left(  f(x,m(x) + \eta) + K_Jx  \right)  \\
\displaystyle \dot y &=& {\varepsilon}^{-1} g(x,m(x) + \eta) + \Delta( m(x) + \eta) 
\end{array}
\right. .
\end{equation}
Thus
\begin{eqnarray*}
&& \frac{d}{dt} \left( \frac{1}{2} \int_\Omega \eta^2 \, du\right) =  \int_\Omega \eta \, \dot \eta \, du \\    
&=& \int_\Omega \frac{\eta}{\varepsilon} \left( g(x,m(x)+\eta) - g(x,m(x)) \right)  du \\
& &  + \int_\Omega \eta \Delta(m(x) +\eta) \, du - \int_\Omega \eta m'(x) \left(  f(x,m(x)+\eta) + K_Jx  \right) du \\
&=& I_1 + I_2 - I_3.
\end{eqnarray*}
Now, let us evaluate each one of the integrals $I_i$ for $i=1, 2, 3$.

Since $g$ is smooth and satisfies (iii) at ${\bf (H_{fg})}$, 
there exists $b=b(u)$ such that
$$
I_1 = \int_\Omega \frac{\eta^2}{\varepsilon} \frac{\partial g}{\partial y}(x,b) \, du
\leq - \frac{\delta}{\varepsilon}  \int_\Omega \eta^2 \, du
\leq - \frac{\delta}{\varepsilon} \| \eta \|_{L^2(\Omega)}^2.
$$

Next, integrating by parts, and using boundary condition \eqref{bcond} and hypotheses ${\bf (H_{m})}$, we obtain by Young's inequality  that 
\begin{eqnarray*}
I_2 &=& \int_\Omega \eta \Delta( m(x) + \eta ) \, du \\
&=& - \int_\Omega |\nabla \eta|^2 \, du - \int_\Omega \nabla \eta \cdot \nabla(m(x)) \, du \\
&\leq& - \int_\Omega |\nabla \eta|^2 \left(  1- \frac{\xi^2}{2}  \right) du + \frac{1}{2 \xi^2} \int_\Omega |\nabla(m(x))|^2 \, du \\
&\leq& \frac{1}{2 \xi^2} \int_\Omega |m'(x)|^2 \, | \nabla x |^2 du \\
&\leq&  \frac{1}{2} \left(\frac{\gamma M}{\xi \delta}\right)^2 \| \nabla x \|_{L^2(\Omega)}^2,
\end{eqnarray*}
if $\xi > \sqrt{2}$.

Also, since $f$ is a smooth function and satisfies (ii) at ${\bf (H_{fg})}$, we have that   
there exists $b$ with $m(x) \leq b \leq m(x) + \eta$, without lost of generality, such that 
\begin{eqnarray*}
I_3 &=& \int_\Omega \eta m'(x) \left[ f(x,m(x) + \eta) - f(x,m(x)) + f(x,m(x)) + K_Jx \right]   du \\
&=& \int_\Omega \eta m'(x) \left[ \frac{\partial f}{\partial y}(x,b)  \eta + f(x,m(x)) + K_Jx \right] du \\
&\geq& \int_\Omega \eta m'(x) \left[ f(x,m(x)) + K_Jx \right] du.
\end{eqnarray*}
Hence, since $|f(x,y)| \leq K$ for all $0\leq x \leq N$ and $0 \leq y \leq M$, and 
\begin{eqnarray*}
&&\left| \int_\Omega y \, K_Jx \, du \right|  =  \left|  \int_\Omega y(u) \int_\Omega J(u-v) ( x(v) - x(u) ) \, dv \, du \right|   \\
& \leq & \left| \int_\Omega y(u) \int_\Omega J(u-v) \, x(v) \, dv \, du \right| + \left| \int_\Omega y(u) \, x(u) \int_\Omega J(u-v) \, dv \, du \right| \\
& \leq & \| x \|_{L^2(\Omega)} \| y \|_{L^2(\Omega)} \left(  |\Omega| \, \| J \|_{\infty} + 1 \right),
\end{eqnarray*}
we obtain from ${\bf (H_{m})}$
\begin{eqnarray*}
- I_3 & \leq & \frac{\gamma M}{\delta} \int_\Omega |\eta| \left[ |f(x,m(x))| + |K_Jx| \right] du \\
& \leq & \frac{\gamma M}{\delta}  \| \eta \|_{L^2(\Omega)} \left[  k |\Omega|^{1/2}   +  \| x \|_{L^2(\Omega)} \left(  |\Omega| \, \| J \|_{\infty} + 1 \right)  \right].
\end{eqnarray*}

Thus, we get that 
\begin{eqnarray*}
\frac{d}{dt} \left( \frac{1}{2} \int_\Omega \eta^2 \, du\right) &=&  I_1 + I_2 - I_3 \\    
& \leq &  - \frac{\delta}{\varepsilon} \| \eta \|_{L^2(\Omega)}^2 + \frac{1}{2} \left(\frac{\gamma M}{\xi \delta}\right)^2 \| \nabla x \|_{L^2(\Omega)}^2 \\
& & + \frac{\gamma M}{\delta}  \| \eta \|_{L^2(\Omega)} \left[  k |\Omega|^{1/2}   +  \| x \|_{L^2(\Omega)} \left(  |\Omega| \, \| J \|_{\infty} + 1 \right)  \right].
\end{eqnarray*}

Now, it follows from Lemma \ref{UB} that norms 
$\| x \|_{L^2(\Omega)}$ and $\| \nabla x \|_{L^2(\Omega)}$
are uniformly bounded in any bounded interval of time as $[0,T]$.
Thus, there exist positive constants $D_0$ and $D_1$ such that 
\begin{eqnarray*}
\frac{d}{dt} \left( \| \eta \|_{L^2(\Omega)}^2 \right) 
& \leq & - \frac{2 \delta}{\varepsilon} \| \eta \|_{L^2(\Omega)}^2 + D_0 \| \eta \|_{L^2(\Omega)} + D_1 \quad \forall t \in [0,T]. 
\end{eqnarray*}
Consequently, by Young's inequality we get 
\begin{eqnarray*}
\frac{d}{dt} \left( \| \eta \|_{L^2(\Omega)}^2 \right) 
& \leq & - \frac{ \delta}{\varepsilon} \| \eta \|_{L^2(\Omega)}^2 + \frac{\varepsilon}{2 \delta} D_0^2 + D_1. 
\end{eqnarray*}

Hence, if we integrate this inequality in $[0,t]$, we can conclude that   
\begin{equation}  \label{etaeq}
\| \eta \|_{L^2(\Omega)}^2 \leq \frac{\varepsilon}{\delta} \left( \frac{\varepsilon}{2 \delta} D_0^2 
+ D_1 \right) \left(1-e^{-t \delta/\varepsilon}   \right)  +  e^{-t \delta/\varepsilon} \| \eta(0) \|_{L^2(\Omega)}^2  
\end{equation}
for all $t \in [0,T]$.

Let us estimate now the convergence of the functions $x$ given by the system \eqref{eq} and \eqref{siseta} to the solution $X$ of the limit equation \eqref{limit} as $\varepsilon$ goes to zero.
We proceed as before considering  
\begin{eqnarray*}
\frac{d}{dt} \left(  \frac{1}{2} \int_\Omega | x - X |^2 \, du \right)  & = & \int_\Omega (x -X)(x-X)' \, du \\
& = & \int_\Omega (x-X)(f(x,m(x)+\eta) - f(X,m(X)) ) \, du  \\
& & \qquad \qquad + \int_\Omega (x-X)K_J(x-X) \, du. \\
& = & I_1 + I_2.
\end{eqnarray*}

First, we evaluate $I_1$ by 
\begin{eqnarray*}
&& I_1   =   \int_\Omega (x-X)(f(x,m(x)+\eta) - f(X,m(X)) ) \, du \\
& = & \int_\Omega (x-X)(f(x,m(x)+\eta) - f(x,m(x)) + f(x,m(x)) - f(X,m(x)) \\ 
& & \quad \qquad + f(X,m(x)) - f(X,m(X)) ) \, du \\
& = & \int_\Omega (x - X) \left(  \frac{\partial f}{\partial y}(x, b_1) \eta  +   \frac{\partial f}{\partial x}(a_1, m(x))(x-X) + \frac{\partial f}{\partial y}(X, b_2)(m(x) - m(X))       \right) du  \\
& = & \int_\Omega \frac{\partial f}{\partial y}(x, b_1) (x - X)  \eta \, du + \int_\Omega \left( \frac{\partial f}{\partial x}(a_1, m(x)) + \frac{\partial f}{\partial y}(X, b_2)m'(a_2) \right) (x-X)^2  \,  du
\end{eqnarray*}
for some $b_1 \in [m(x),m(x)+\eta]$,  $b_2 \in [m(x),m(X)]$ and $a_1$, $a_2 \in [x, X]$ (without lost of generality).

Thus, by Young's Inequality and conditions ${\bf (H_{fg})}$, we have for an appropriated $\rho > 0$ that 
\begin{eqnarray*}
I_1 & \leq & \left[ \beta N \left( \frac{\gamma M}{\delta}  + \rho^2 \right) - \alpha \right] \int_\Omega (x-X)^2 \, du  + \frac{\beta N}{\rho^2} \int_\Omega \eta^2 \, du.
\end{eqnarray*}

On the other hand, the nonlocal operator satisfies 
\begin{eqnarray*}
I_2 & = & \int_\Omega (x-X) K_J (x-X) \, du \\
& = & \int_\Omega (x-X)(u) \int_\Omega J(u-v) \left( (x-X)(v) - (x-X)(u) \right) dv \, du  \\
& = & - \int_\Omega \int_\Omega J(u-v) \left( (x-X)(v) - (x-X)(u) \right)^2 dv \, du \\
& \leq & 0.
\end{eqnarray*}

Therefore, if we set the constants 
\begin{equation} \label{eqconst}
C_1= \beta N \left( \frac{\gamma M}{\delta}  + \rho^2 \right) - \alpha 
\quad \textrm{ and } \quad 
C_2=\frac{\beta N}{\rho^2}
\end{equation}
we get 
\begin{equation} \label{eqx-X}
\frac{d}{dt} \left(  \int_\Omega | x - X |^2 \, du \right) \leq 2 C_1 \int_\Omega (x-X)^2 \, du  + 2 C_2 \int_\Omega \eta^2 \, du.
\end{equation}
Consequently, we get after some integrations that 
\begin{equation}
\| x - X \|_{L^2(\Omega)}^2 \leq e^{2 C_1 t} \| (x-X)(0) \|^2_{L^2(\Omega)} + 2 C_2 \int_0^t e^{2 C_1(t-s)} \| \eta \|_{L^2(\Omega)}^2 \, ds.
\end{equation}
Thus, due to \eqref{etaeq}, we can conclude that
\begin{eqnarray*}
\| x - X \|_{L^2(\Omega)}^2 & \leq & e^{2 C_1 t} \| (x-X)(0) \|^2_{L^2(\Omega)}   \\ 
& & + 2 \varepsilon C_2 e^{2 C_1 t} \left[  \left( \frac{\varepsilon D_0^2 + 2 \delta D_1}{2 \delta^2} \right) \left( \frac{1 - e^{-2 C_1 t}}{2 C_1} \right) \right. \\
&& \qquad \left. +  \frac{\| \eta(0) \|_{L^2(\Omega)}^2}{2 C_1 \varepsilon + \delta} \left( 1- e^{-\frac{t}{\varepsilon}(2 C_1 \varepsilon + \delta)} \right) \right].
\end{eqnarray*}

Therefore, since we are taking $\varepsilon>0$ small, we obtain 
\begin{eqnarray}  \label{SEST}
\| x - X \|_{L^2(\Omega)}^2 &\leq & e^{2 C_1 t} \| (x-X)(0) \|^2_{L^2(\Omega)} \nonumber \\
& & + 2 \varepsilon C_2 e^{2 C_1 t} \left[  \left( \frac{\varepsilon D_0^2 + 2 \delta D_1}{4 \delta^2 |C_1|} \right) 
+  \frac{\| \eta(0) \|_{L^2(\Omega)}^2}{2 C_1 \varepsilon + \delta} \right].
\end{eqnarray}

As a consequence, we have the following results:

\begin{theorem} \label{propE1}
Let us suppose assumptions ${\bf (H_{fg})}$, ${\bf (H_{\infty})}$, ${\bf (H_{m})}$ and ${\bf (H_{J})}$ with the additional condition $J$ of class $\mathcal{C}^1$. 

Then, for any $T>0$ and initial condition $(x_0, y_0) \in \mathcal{C}^1(\Omega) \times H^1(\Omega)$ satisfying $0 \leq \|x_0\|_{L^\infty(\Omega)} \leq N$ and $0 \leq \|y_0\|_{L^\infty}(\Omega) \leq M$, 
there exist positive constants $\varepsilon_0$, $M_1$ and $M_2$ such that solution $x$ of \eqref{equa} satisfies 
\begin{eqnarray*}
\| x - X \|_{L^2(\Omega)} \leq M_1 \| (x-X)(0) \|_{L^2(\Omega)}  + \varepsilon \, M_2 
\end{eqnarray*}
for all $t \in [0, T]$ and $\varepsilon \in (0,\varepsilon_0)$ where $X$ is the solution of the limit problem \eqref{limit} with initial condition $X(0)$.

In particular, if $x(0) = X(0)$, we have 
$$
\sup_{t \in [0,T]} \| x - X \|_{L^2(\Omega)} \to 0 \qquad \textrm{ as } \varepsilon \to 0.      
$$
\end{theorem}
\begin{proof}
It is a direct consequence of estimate \eqref{SEST}.
\end{proof}

\begin{remark}
Since the constants $D_0$ and $D_1$ given by Lemma \ref{UB} depend on $T$, we can not guarantee convergence at $(0,+\infty)$. 
\end{remark}

\begin{cor} \label{decay}
Under the conditions of Theorem \ref{propE1} and the additional assumption 
\begin{equation} \label{acond}
\frac{\alpha}{\beta N} - \frac{\gamma M}{\delta} > 0,
\end{equation}
there exist positive constants $\varepsilon_0$, $m$ and $M_1$ such that
\begin{equation} \label{estadd}
\| x - X \|_{L^2(\Omega)}^2 \leq e^{-m t} \left( \| (x-X)(0) \|^2_{L^2(\Omega)}  + \varepsilon \, M_2 \right) 
\end{equation}
for all $t \in [0, T]$ and $\varepsilon \in (0,\varepsilon_0)$.

In particular, if $x(0) = X(0)$, we have 
$$
\sup_{t \in [0,T]} \| x - X \|_{L^2(\Omega)} \to 0 \qquad \textrm{ as } \varepsilon \to 0.      
$$  
\end{cor}
\begin{proof}
Due to the additional condition \eqref{acond}, we can choose $\rho$ small enough in order to set $C_1<0$ in expression \eqref{eqconst}. Hence, we obtain estimate \eqref{estadd} from inequality \eqref{SEST} concluding the proof.
\end{proof}

\begin{remark}
Notice that, due to Remark \ref{0stable}, the condition \eqref{acond} implies that zero is the globally stable equilibrium for the non negative solutions of the limit equation \eqref{limit}. Thus, we are just giving an order of decaying at Corollary \ref{decay}.  
\end{remark}

\section{Application} \label{appl}

Coming back to (\ref{SIRUV}),(\ref{Neuman}) it is easy to check that Theorem~\ref{propE1} and Corollary~\ref{decay} can be applied. So, we have uniform convergence in finite intervals of time. However, both systems, namely (\ref{SIRUV}),(\ref{Neuman}) and (\ref{limin}) have the same equilibria which are always constant in space, the disease free (i=j=0) and an non-zero equilibria which will be the endemic equilibrium (for $R_0>1$, the basic reproduction number).
In this later case, the constant endemic equilibrium  is locally stable, that can be check through linealization, similar to the ODE case.

Furthermore, since we have comparison, we always can compare the constant in space solutions (which will satisfy the ODE system) with the solutions of (\ref{SIRUV}),(\ref{Neuman}). But the ODE system have an equivalent singular perturbation result, with the difference that the convergence can be done, globally in time. Therefore, we can extend this results for all times.

\section{Appendix} 

Here we show that the solutions $(x,y)$ of \eqref{eq} are uniformly bounded in $H^1(\Omega)$ for any finite interval of time.
Notice that the boundedness in $L^\infty(\Omega)$ follows from Remark \ref{Linfty}, and then, $x$ and $y$ are uniformly bounded in $L^2(\Omega)$.
It remains us to estimate $\nabla x$ and $\nabla y$ in $L^2(\Omega)$.

In order to estimate $\|\nabla y\|_{L^2(\Omega)}$, we first perform the change of variable $t=\epsilon \tau$ in \eqref{equa} obtaining 
$$
\dot w = g(z,w) + \epsilon \Delta w  \textrm{ in } \Omega \quad \textrm{ with } \quad \frac{\partial w}{\partial N}=0 \textrm{ on } \partial \Omega
$$
where $w(\tau) = y(\tau \epsilon)$ and $z(\tau) = x(\tau \epsilon)$.

Hence, if we define the norm $\| w \|_s = \| (\epsilon \Delta + I)^s w \|_{L^2(\Omega)}$ for any $s \geq 0$ and $\epsilon>0$ we get from \cite[Theorem 1.4.3]{Henry} that
$$
\| e^{(\epsilon \Delta - I) \tau} \|_s \leq M e^{-\tau} \tau^{-s} \quad \textrm{ wherever }\tau > 0 \textrm{ and } 0< s \leq 1,
$$
since the first eigenvalue of $\epsilon \Delta - I$ is equal to $1$ for any $\epsilon>0$. Thus, due to \cite[Theorem 3.3.6]{Henry} and assumptions on nonlinearity $g$, we have for any $0<s<r\leq1$ that 
$\| w \|_r$ is uniformly bounded for $\tau>1$ and any $\epsilon>0$. Therefore, we get that 
$\| y(t) \|_r = \| w(t/\epsilon)\|_r$ is uniformly bounded wherever $t>\epsilon$.  As we are taking $\epsilon \to 0$, we can conclude that $\| y \|_{H^1(\Omega)}$ is uniformly bounded for any $t>1$ and $\epsilon \in (0,1)$.

Now, let us estimate $\|\nabla x\|_{L^2(\Omega)}$.  We use constant variation formula at the first equation of \eqref{eq} getting 
$$
x(t) = e^{-At}x_0 + \int_0^t e^{-A(t-s)} \left( f(x,y) + \int_\Omega J(u-v) x(s) \, dv \right) ds \quad \textrm{ in } \Omega
$$
with $A(u) = \int_\Omega J(u-v) \, dv$ for $u \in \Omega$.
Hence, under the conditions $x_0$ and $J$ of class $\mathcal{C}^1$, we get 
\begin{eqnarray*}
\partial_i x(t) & = & e^{-At} \left( \partial_i x_0 - \partial_i A \, x_0 \right) \\ 
&& - \partial_i A \int_0^t (t-s) e^{-A(t-s)} \left(   f(x,y) + \int_\Omega J(u-v)x(s) \, dv   \right) ds \\
& & + \int_0^t e^{-A(t-s)} \left( \nabla f \cdot (\partial_i x,\partial_i y) + \int_\Omega \partial_i J(u-v) x(s) \, dv \right) ds 
\end{eqnarray*}
where $\partial_i$ denotes the $i$-th partial derivative for $i=1, 2, ..., n$.

Thus, since $m=\min_{u \in \Omega} A(u) > 0$ with $x$ and $y$ uniformly bonded in $L^\infty(\Omega)$, we obtain that there exist positive constants $C_j$ such that
\begin{eqnarray*}
e^{2mt} | \partial_i x |^2 & \leq & C_0 + C_1 t \int_0^t (t-s)^2 e^{2ms} ds + C_2 t \int_0^t e^{2ms} ds \\
& & \qquad + C_3 t \int_0^t e^{2ms} |\partial_i y|^2 ds + C_4 t \int_0^t e^{2ms} |\partial_i x|^2 ds.
\end{eqnarray*}
Hence, as $\| y \|_{H^1(\Omega)}$ is uniformly bounded, we obtain for all $t \geq 1$ that 
\begin{eqnarray*}
e^{2mt} \| \partial_i x \|^2_{L^2(\Omega)} & \leq & \tilde C_0 + \tilde C_1 t \left[ \int_0^t (t-s)^2 e^{2ms} ds \right. \\
& & \left. + \int_0^t e^{2ms} ds  +  \int_0^t e^{2ms} \|\partial_i x\|^2_{L^2(\Omega)} ds \right]
\end{eqnarray*}
for some positive constants $\tilde C_0$ and $\tilde C_1$. From Gronwall inequality, we conclude $\| \partial_i x \|^2_{L^2(\Omega)}$ is uniformly bounded in $[0,T]$ which leads us to the following result.

\begin{lemma} \label{UB}
Under assumptions ${\bf (H_{fg})}$ and ${\bf (H_\infty)}$ with $J$ of class $\mathcal{C}^1$ satisfying ${\bf (H_J)}$, we have that, for any given $T>0$, there exists $M>0$ such that 
the solutions $(x,y)$ of \eqref{equa} with initial conditions $x_0 \in \mathcal{C}^1$ and $y_0 \in L^\infty(\Omega) \cap H^1(\Omega)$ satisfy
$$
\| j \|_{H^1(\Omega)} \leq M \quad \textrm{ for all } t \in [0,T]
$$
where $j=x$ or $y$.
\end{lemma}

\vspace{0.7 cm}

{\bf Acknowledgements.} 

\noindent This work does not have any conflicts of interest. The second author (MCP) is partially supported by CNPq 303253/2017-7 and FAPESP 2017/02630-2 and 2019/06221-5 (Brazil).


\begin{thebibliography}{00}






\bibitem{liu2016climate} Liu-Helmersson, J., Quam, M., Wilder-Smith, A., Stenlund, H., Ebi, K., Massad, E., and Rockl{\"o}v, J.,\textit{Climate change and {A}edes vectors: 21st century projections for dengue transmission in {E}urope}, EBioMedicine, 7, p. 267-277, 2016.

\bibitem{wang2017} Wang, L., Zhao, H., Oliva, S. M., and Zhu, H., \textit{Modeling the transmission and control of {Z}ika in {B}razil}, Scientific reports, 7(1), 7721, 2017.

\bibitem{massad2008scale} Massad, E., Ma, S., Chen, M., Struchiner, C. J., Stollenwerk, N., and Aguiar, M., \textit{Scale-free network of a dengue epidemic}. Applied Mathematics and Computation, 195(2), p. 376-381, 2008.

\bibitem{boccia2014} Boccia, T. M. Q. R., Burattini, M. N., Coutinho, F. A. B., and Massad, E., \textit{Will people change their vector-control practices in the presence of an imperfect dengue vaccine?}, Epidemiology and infection, 142(03), p. 625-633, 2014.

\bibitem{kraemer2015global} Kraemer, Moritz U.G., et al., 
\textit{The global distribution of the arbovirus vectors Aedes aegypti and Ae. albopictus}, Elife, 4:e08347, 2015.


\bibitem{shepard2016global} Shepard, D. S., Undurraga, E. A., Halasa, Y. A., and Stanaway, J. D., \textit{The global economic burden of dengue: a systematic analysis}, The Lancet infectious diseases, 16(8), p. 935-941, 2016.

\bibitem{rodriguez2011re} Rodriguez-Barraquer, I., Cordeiro, M. T., Braga, C., de Souza, W. V., Marques, E. T., and Cummings, D. A., \textit{From re-emergence to hyperendemicity: the natural history of the dengue epidemic in {B}razil}, PLoS neglected tropical diseases, 5(1), e935, 2011.

\bibitem{JKing2018} King, J. G., Souto-Maior, C., Sartori, L. M., Maciel-de-Freitas, R., and Gomes, M. G. M., \textit{Variation in Wolbachia effects on Aedes mosquitoes as a determinant of invasiveness and vectorial capacity}, Nature Communications, 9(1), 2018. 

\bibitem{amaku2014} Amaku, M., Coutinho, F. A. B., Raimundo, S. M., Lopez, L. F., Burattini, M. N., and Massad, E., 
\textit{A comparative analysis of the relative efficacy of vector-control strategies against dengue fever}, Bulletin of mathematical biology, 76(3), p. 697-717, 2014.

\bibitem{amaku2016magnitude} Amaku, Marcos, et al.
\textit{Magnitude and frequency variations of vector-borne infection outbreaks using the {R}oss--{M}acdonald model: explaining and predicting outbreaks of dengue fever},
Epidemiology \& Infection, 144(16), p. 3435-3450, 2016.

\bibitem{dosSantos2018}
dos Santos, B. C., Sartori, L. M., Peixoto, C., Bevilacqua, J. S., and Oliva, S. M., \textit{Prospective Study About the Influence of Human Mobility in Dengue Transmission in the State of {R}io de {J}aneiro}, In Modeling, Dynamics, Optimization and Bioeconomics III, Springer, Cham, p. 419-427,  2018.

\bibitem{rocha2013time} Rocha, F., Aguiar, M., Souza, M., and Stollenwerk, N., \textit{Time-scale separation and centre manifold analysis describing vector-borne disease dynamics},   
International Journal of Computer Mathematics, 90(10), p. 2105-2125, 2013.

\bibitem{iggidr2017vector} Iggidr, A., Koiller, J., Penna, M. L. F., Sallet, G., Silva, M. A., and Souza, M. O., \newblock \textit{Vector borne diseases on an urban environment: the effects of heterogeneity and human circulation},  Ecological complexity, 30, p. 76-90, 2017. 

\bibitem{precioso2015clinical} Precioso, A. R., Palacios, R., Thom\'e, B., Mondini, G., Braga, P., and Kalil, J., 
\textit{Clinical evaluation strategies for a live attenuated tetravalent dengue vaccine}, Vaccine, 33(50), p. 7121-7125, 2015.


\bibitem{maier2017analysis} Maier, S. B., Huang  X., Massad, E., Amaku, M., Burattini M. N., and Greenhalgh, D.,
\textit{Analysis of the optimal vaccination age for dengue in Brazil with a tetravalent dengue vaccine}. Mathematical biosciences, 294, p. 15-32, 2017.

\bibitem{massadestimating} Massad, E. et al., \textit{Estimating the size of Aedes aegypti populations from dengue incidence data: Implications for the risk of yellow fever outbreaks}. Infectious Disease Modelling, v. 2, p. 441-454, 2017.

\bibitem{magal} Ducrot, A., Magal, P. and Seydi, O., \textit{Singular perturbation for an abstract non-densely defined Cauchy problem}, J. Evol. Equ., 17, p. 1089?1128, 2017.


\bibitem{Henry} Henry, D. B., \emph{Geometric Theory of Semilinear Parabolic Equations}, Lecture Notes in Math., vol 840, Springer-Verlag, 1981.

\bibitem{hale} Hale, J. K., \emph{Asymptotic Behaviour of Dissipative Systems}, Math. Surveys and Monographs, 25, 1998.








\bibitem{libro}  Andreu-Vaillo, F., Maz\'{o}n, J. M., Rossi, J. D.,  and Toledo, J., \emph{Nonlocal Diffusion Problems}.   Mathematical Surveys and Monographs, vol. 165. AMS, 2010.

\bibitem{Fife} Fife, P. C.,  \emph{Mathematical aspects of reaction-diffusion equations}. Lect. Notes in Math. 28, Springer-Verlag New-York, 1979.

\bibitem{HMMV} Hutson, V., Martinez, S., Mischaikow, K., and Vickers, G. T., \emph{The evolution of dispersal}. J. Math. Biol. 47, p. 483-517, 2003.

\bibitem{Pao} Pao, C. V.,  \emph{Nonlinear parabolic and elliptic equations}. Plenum Press, New York, 1992. 

\bibitem{marconelarissaoliva} Sartori, L.M., Pereira, M. C.,  and Oliva, S., \emph{Parameter fitting using Time-scale analysis for vector borne diseases with spatial dynamics}, bioRxiv doi.org/10.1101/759308. 

\bibitem{Ruan}  Ruan S. and Wang, W., \emph{Dynamical behavior of an epidemic model with a
nonlinear incidence rate}, J. Differential Equations 188, p. 135?163, 2003.

\end{thebibliography}



\end{document}